\documentclass[final,3p]{elsarticle}

\usepackage[normalem]{ulem}
\usepackage{tikz}
\usepackage{epstopdf}

\usepackage{graphicx}
\usepackage{color}
\usepackage{amsthm}
\usepackage{graphics}
\usepackage{amsmath,amssymb,amsfonts}
\usepackage{epsfig}
\usepackage{subfigure}
\usepackage{multirow}

\biboptions{sort&compress}

\begingroup 
\newtheorem{thm}{Theorem}[section]   
\newtheorem{lem}[thm]{Lemma}         
\newtheorem{prop}[thm]{Proposition}  

\endgroup

\theoremstyle{definition}
\newtheorem{defi}[thm]{Definition}   

\theoremstyle{remark}

\newcommand\comm[1]{}  

\renewcommand{\phi}{\varphi}

\newcommand{\R}{{\mathbb R}}
\newcommand{\C}{{\mathbb C}}

\newcommand\e{\varepsilon}

\renewcommand\phi{\varphi}

\newcommand\by{\mbox{$\times$}}
\numberwithin{equation}{section}

\begin{document}

\date{}

\title{Continuous products of matrices\footnotetext[1]{The research was carried out at the IITP RAS at the expense of the Russian
 Foundation for Sciences (project № 14-50-00150).}}

\author[av]{Alexander Vladimirov}
\ead{vladim@iitp.ru}

\address[av]{Institute for Information Transmission Problems
\\Russian Academy of Sciences}

\begin{abstract}
We answer the question if the continuous product of square matrices $M(t)$
over $t\in [0,1]$ can be correctly defined. The case where all $M(t)$ are
taken from a finite set $\Sigma$ is studied. We find necessary and
sufficient conditions on $\Sigma$ that ensure the convergence of products
$M(t_{0}=0)M(t_{1})\dots M(t_{N}=1)$ as the partition $0<t_{1}<\dots<1$
refines. These conditions are properties LCP (left convergent product) and
RCP (right convergent product) of the set $\Sigma$. That is, it suffices to
require the convergence of all finite products $M_{1}M_{2}\dots M_{K}$ and
$M_{K}\dots M_{2}M_{1}$ as $K\to\infty$, where $M_{i}\in\Sigma$. The theory
of joint spectral radius is heavily used.
\end{abstract}

\maketitle

\section{Introduction}

Directional derivatives of polyhedral {\em sweeping processes}
\cite{MorEv7734,KunAn00,KrePo0391,10.2307/40800846} are closely related to
finite sets of projection matrices and their infinite products taken in
special order. Here we develop the theory of {\em continuous products} of
matrices keeping in mind its possible applications to sweeping process and
similar hysteresis systems.

Suppose we want to define correctly the ``multiplicative integral" of a
matrix-valued function $M(t)\in \Sigma\subseteq M_{n}(\R)$ (or $M_{n}(\C)$)
over the time interval $[0,1]$. That is, the question is if there exists a
limit of finite products $M(t_{1})\dots M(t_{K})$ as the partitions $0\le
t_{1}<\dots<t_{K}\le 1$ of $[0,1]$ refine in the sense of inclusion.

An obvious case where the answer is positive is that of a finite family
$\Sigma$ such that $\rho(\Sigma)<1$, where $\rho(\Sigma)$ is the {\em joint
spectral radius}, see
\cite{RotA6037,DauSe9222,BerBo9221,BloTh2008,Jun09,Koz:LAA10}. It is not
surprising that each function $M:[0,1]\to \Sigma$ has the zero continuous
product over $[0,1]$ (or over any other infinite linearly ordered set).

There are less obvious cases where the limit exists, again, for each map
$M:[0,1]\to \Sigma$ (no continuity or any other regularity is needed) and may
be different from zero. As we prove in this paper, for finite families
$\Sigma$, this happens exactly if $\Sigma$ is both of {\em left convergent
products (LCP)} and {\em right convergent products (RCP)} type.

Additionally, we demonstrate that LCP and RCP properties together imply {\em
transversality} (TR) of the family $\Sigma$ and, moreover, that LCP and TR
imply RCP and that RCP and TR imply LCP. It happens also that LCP and RCP
imply the convergence of any infinite sequence of matrix products where each
subsequent product is obtained by insertion of an arbitrary matrix
$A\in\Sigma$ in an arbitrary position of the previous product (CP property).

\section{Definitions}

Let a finite set ${\Sigma}=\{A_j: j\in J=\{1,\dots,k\}\}$ of real $n\by
n$-matrices be given. Denote by $\Sigma^m$, $m=1,2,\dots$, the set of all
products $A_{j_1}\dots A_{j_m}$, $A_{j_i}\in\Sigma$ for $i=1,\dots,m$. Denote
also $\Sigma^0=\{I\}$, where $I$ is the identity $n\by n$-matrix. Let
$$
{\mathcal L}(\Sigma)=\cup_{m=0,1,\dots} \Sigma^m.
$$

The set $\Sigma$ is said to be {\em product bounded} if there exists a $C>0$
such that $\|A\|<C$, $A\in {\mathcal L}(\Sigma)$.

\begin{defi}\rm
A set $\Sigma$ is called {\em LCP} (left convergent products)
if, for any sequence $S=\{j_i\in J:\ i=1,2,\dots\}$, there exists
a limit matrix $L_S$ such that
$$
\lim_{m\to\infty}\|L_S^m-L_S\|=0,
$$
where
\begin{equation}\label{elcp}
L_S^m=A_{j_m}\dots A_{j_1},\quad m=1,\dots.
\end{equation}

The set ${\Sigma}$ is called {\em RCP} if, for any sequence $S=\{j_i\in J:\
i=1,2,\dots\}$, there exists a limit matrix $R_S$ such that
$$
\lim_{m\to\infty}\|R_S^m-R_S\|=0,
$$
where
\begin{equation}\label{ercp}
R_S^m=A_{j_1}\dots  A_{j_m},\quad m=1,\dots.
\end{equation}
\end{defi}

These properties are not the same as the following simple example
demonstrates. Let $\Sigma=\{A_1,A_2\}$ in $\R^2$, where
$$
A_1=\left(\begin{array}{cc}
1&0\\0&0
\end{array}\right),\ \ %
A_2=\left(\begin{array}{cc}
1&1\\0&0
\end{array}\right).
$$
Since $A_1 A_2=A_2$ and $A_2 A_1=A_1$, we get $A_{j_1}\dots A_{j_m}=A_{j_m}$
for any sequence $j_i\in\{1,2\}$, $i=1,\dots,m$, and the set $\Sigma$ is LCP
and not RCP.

Both properties LCP and RCP are stronger than the product boundedness
\cite{BerBo9221}. Transposing the
matrix equality (\ref{elcp}) we get the following assertion.
\begin{prop}
A set $\Sigma=\{A_j: j\in J\}$ is LCP if and only if the dual set
$\Sigma^*=\{A_j^*: j\in J\}$ is RCP.
\end{prop}

By {\em discrete linear inclusion} DLI($\Sigma$) we will understand the set
of all infinite sequences $\{x_i\}$, $i=0,1,\dots$, of vectors in $\R^n$ such
that
\begin{equation}\label{ko5}
x_i=A_{j_i} x_{i-1},\quad j_i\in J,\quad i=1,2,\dots.
\end{equation}
These sequences will be called {\em paths} of $\Sigma$. The LCP property is
equivalent to convergence of any path $\{x_i\}$ of $\Sigma$, see
\cite{BerBo9221}.

There are equivalent definitions of LCP and RCP for finite sets $\Sigma$,
see, for instance, \cite{VlaSt9912}. The family $\Sigma$ is LCP if and only
if all its paths have bounded variation. It is RCP if and only if each family
of affine maps
\[
\Psi=\{A+(A-E)h_{A}: A\in\Sigma,\ h_{A}\in\R^{n}\}
\]
generates a bounded semigroup. They are {\em contraction semigroups} if
$\Sigma$ is irreducible. For affine semigroups see \cite{1064-5616-206-7-921}
and the bibliography within.

For any subset $J'\subseteq J$ of indices, let us define
two subspaces of $\R^n$:
$$
N_{J'}=\bigcap_{j\in J'} N(I-A_j),\quad
R_{J'}={\rm span}_{j\in J'} \{R(I-A_j)\},
$$
where $N(A)=\{x\in\R^n: Ax=0\}$ is the nullspace of $A$
and $R(A)=\{Ax:x\in\R^n\}$ is the range of $A$.

Both $N_{J'}$ and $R_{J'}$ are invariant subspaces for paths of the subset
$\Sigma'=\{A_j\in \Sigma: j\in J'\}$. Moreover, any affine subspace of the
form $x+R_{J'}$, $x\in\R^n$, is also invariant for these paths because
$x_i-x_{i-1}\in R_{J'}$ whenever $\{x_i\}$ is a path of $\Sigma'$.

\begin{defi}\rm
The set $\Sigma$ is {\em $0$-transversal} if $N_{J'}\cap R_{J'}=0$
for each $J'\subseteq J$; it is {\em $\R^n$-transversal} if
$N_{J'}+R_{J'}=\R^n$ for each $J'\subseteq J$.
The set $\Sigma$ is {\em transversal} if $N_{J'}\oplus R_{J'}=\R^n$
for each $J'\subseteq J$,
that is, if it is both $0$-transversal and $\R^n$-transversal.
\end{defi}

Obviously, the properties of $0$-transversality and $\R^n$-transversality
are dual to each other: any one of them holds for $\Sigma$ if and only
if the other one holds for $\Sigma^*=\{A^*:A\in\Sigma\}$.
Let us prove a simple auxiliary assertion.

\begin{lem}\label{l32}
If $\Sigma$ is LCP then $\Sigma$ is $\R^n$-transversal.
\end{lem}
\begin{proof}
Suppose the contrary, that is, $N_{J'}+R_{J'}\ne\R^n$ for some $J'\subseteq
J$. Then there exists an $x_0\in\R^n$ such that $(x_0+R_{J'})\cap
N_{J'}=\emptyset$. Consider a path $\{x_0,x_1,\dots\}$ of $\Sigma$ where each
matrix $A_j$, $j\in J'$ is used infinitely many times. Then the limit point
$x^*=\lim_{i\to\infty}x_i$ belongs to $N_{J'}$. Since $x_i\in x_0+ R_{J'}$,
we also have $x^*\in x_0+R_{J'}$, which is a contradiction.
\end{proof}

Hence, the dual assertion is also true:
\begin{lem}
If $\Sigma$ is RCP then $\Sigma$ is $0$-transversal.
\end{lem}

Thus, we get the following result.

\begin{prop}\label{t1}
If $\Sigma$ is both LCP and RCP then it is transversal.
\end{prop}

\section{Relations between LCP and RCP}

We will need the following assertion from \cite{ElsNo9733}.

\begin{thm}\label{bels}
Suppose $\Sigma=\{A_j:j\in J\}$ is LCP and $N_J=\{0\}$.
Then there exists a norm $\|\cdot\|_{\Sigma}$
in $\R^n$ and  a constant $0\le q<1$ such that
\begin{equation}\label{st1}
\|A_j\|_{\Sigma}\le 1,\ j\in J, \quad{\rm and}
\end{equation}
\begin{equation}\label{st2}
\|A_{j_1}\dots A_{j_m}\|_{\Sigma}\le q
\end{equation}
for any finite product containing each $A_j$ from $\Sigma$.
\end{thm}

\begin{thm}\label{t2}
If $\Sigma=\{A_j:j\in J\}$
is LCP and $0$-transversal then it is also RCP.
\end{thm}
\begin{proof}
Let us use induction on the cardinality $k$ of $\Sigma$.
For $k=1$, there is no difference between LCP and RCP by
definition.

Suppose the assertion is true for all matrix sets of cardinality
$1,\dots,k-1$ and prove it for $\Sigma=\{A_j:j\in J=\{1,\dots,k\}\}$.
Let us first assume $N_J=\{0\}$
Consider a right-infinite product of matrices $A_{j_1} A_{j_2}\dots$.
Suppose that any $j\in J$ occurs in the sequence
$\{j_1,j_2,\dots\}$ infinitely many times. Then we can represent
each finite product $A_{j_1}\dots A_{j_m}$ as
$B_1\dots B_{s(m)}$, where each $B_p$, $p=1,\dots, s(m)$,
is a finite product of matrices $A_{j_i}$ containing all
$A_j$, $j\in J$, and $s(m)\to\infty$ as $m\to\infty$.
Then, according to
Theorem \ref{bels}, the sequence of products $A_{j_1}\dots A_{j_m}$
converges to the zero matrix as $m\to\infty$.

Let us consider the remaining cases, that is, suppose that
for some $j\in J$, there
exists  an $m>0$ such that $j\not\in\{j_m,j_{m+1},\dots\}$.
By the induction hypothesis, the product
$A_{j_m} A_{j_{m+1}}\dots$ converges to some matrix $A$ and,
hence, the product $A_{j_1}\dots A_{j_m}A_{j_{m+1}}\dots$
converges to the matrix $A_{j_1}\dots A_{j_{m-1}}A$,
and the required assertion is proved.

It remains to consider the case $N_J\ne\{0\}$. Because of the
$0$-transversality assumption and Lemma \ref{l32} we
have $R^n=N_J\oplus R_J$. Since the subspaces $R_J$ and $N_J$
are invariant to all $A_j\in\Sigma$,
we can reduce all matrices $A_j$ to the form
$$
A_j=\left(\begin{array}{cc}
I&0\\0&\tilde{A}_j
\end{array}\right)
$$
by the same nonsingular linear change of variables corresponding to
the decomposition $\R^n=N_J\oplus R_J$. Obviously, both the
LCP and RCP properties of $\Sigma$ are equivalent to
those of $\tilde{\Sigma}=\{\tilde{A}_j: j\in J\}$,
and for $\tilde{\Sigma}$ we have $\tilde{N}_J = \{0\}$.
\end{proof}

Transposition of matrices $A_j$ gives us the dual assertion:

\begin{thm}\label{t3}
If $\Sigma=\{A_j:j\in J\}$
is RCP and $\R^n$-transversal then it is also LCP.
\end{thm}

Consider the three properties LCP, RCP and transversality (TR) of
sets of matrices. Proposition \ref{t1} and Theorems
\ref{t2}, and \ref{t3} yield the following result.
\begin{thm}
For a finite set of matrices, any two of the properties LCP, RCP, and TR
imply the third one.
\end{thm}

Let us say that the set $\Sigma$ is {\em CP} if it is LCP, RCP,
and transversal. As an example of a CP set, consider any finite
set $\Sigma$ of orthogonal projections $P_j$ onto subspaces
$L_j\subseteq\R^n$ of arbitrary dimensions. The transversality
of $\Sigma$ is immediate. The LCP property of
$\Sigma$ is an easy consequence of the following result
\cite{VlaSt9912}.

\begin{prop}
If, for any path $\{x_i\}$ of a finite set $\Sigma$, the relation
\begin{equation}\label{6e1}
\lim_{i\to\infty}\|x_{i+1}-x_i\|=0
\end{equation}
holds, then $\Sigma$ is LCP.
\end{prop}
Indeed, for any $j\in J$, we have
$$
\|P_j x\|=\sqrt{\|x\|^2-\|P_j x-x\|^2}
$$
and, hence, (\ref{6e1}) must hold for any path $\{x_i\}$ of $\Sigma$.

\section{Continuous products of matrices}

In all what follows $\Sigma$ is a finite CP set.
Let $R$ be an arbitrary  linearly ordered set and
let $A$ be a map from $R$ to $\Sigma$ (a matrix-valued function
on $R$ ranging in $\Sigma$). For any finite subset
$S=\{s_1<s_2<\dots<s_k\}\subset R$, the product
$$
M_S=A(s_1)\dots A(s_k)
$$
is defined. Since $\Sigma$ is CP, this definition is extended
in a natural way to any countable subset $S\subset R$ of
the form $S=\{s_1<s_2<\dots\}$ or $S=\{s_1>s_2>\dots\}$.
We will define the product
\begin{equation}\label{e31}
M_S=\prod_{r\in S}A(r)
\end{equation}
for each subset $S$ of $R$. It suffices to do this for $R$
itself because each subset of a linearly ordered set is again
a linearly ordered set.

Let us use an iterative procedure on $k=\# J$ for this definition
(by $\#$ we denote the cardinality of the set).
Suppose that, for all $J'$ such that $\#J'<K$,
the products (\ref{e31}) are defined. Define $I_R(A)$ as the
maximal number of disjoint intervals $S_i\subseteq R$
such that $A(S_i)=\Sigma$ (let us call them {\em complete intervals}).
By an interval we understand a set $S\subset R$ without holes,
that is, conditions $a,b\in S$, $g\in R$, $a<g<b$ should imply $g\in S$.
Note that we can assume without loss of generality that
$\cup_i S_i=R$, that is, intervals $S_i$ are {\em adjacent}.
For any incomplete interval $S$, the product $M(S)$ is already
defined.

Now, we need an auxiliary result.
\begin{lem}\label{l3}
If $I_R(A)<\infty$, there exists a partition of $R$ into
no more than $3I_R(A)$ adjacent disjoint incomplete intervals.
\end{lem}
\begin{proof}
First, let $I_R(A)=1$. Let us say that the interval $S$ is a
{\em left interval} of $R$ if
$$
S=\cup_{r\in S}\{x\in R:x\le r\}.
$$
Let us show that there exists an incomplete left interval of $R$.
Suppose the contrary. Then $R$ has no minimal element.
Let us choose a subset $P=\{r_1,\dots,r_k\}\subseteq R$,
$r_1<r_2<\dots <r_k$ such that $A(P)=\Sigma$. The interval
$[r_1,r_2]$ is complete, and the interval $Q=\{x\in R:x<r_1\}$
is also complete, which is a contradiction.

Now, let $S_1$ be the union of all incomplete left intervals.
(it is, obviously, incomplete itself). Denote $R_1=R\backslash S_1$
and repeat the same argument to construct the interval $S_2$ as the
maximal incomplete left interval of $R_1$. If
$R_2\backslash (S_1\cup S_2)$ is nonempty, we repeat the same construction,
and so on. Suppose that $R_3$ is nonempty. Then choose some $r'\in S_2$
and $r''\in R_3$. Obviously, the intervals
$$
I_1=\{x\in R:x\le r'\}\quad{\rm and}\quad
I_2=S_3\cup \{x\in R_3:x\le r''\}
$$
are complete and disjoint. The contradiction finishes the proof
for $I_R(A)=1$. Finally, in the general case $m=I_R(A)<\infty$,
the set $R$ can be partitioned into $m$ disjoint adjacent
complete intervals $Q_i$ such that $I_{Q_i}(A)=1$, $i=1,\dots,m$,
and then each one of these intervals is partitioned into no more
than $3$ disjoint incomplete intervals.
\end{proof}

Let us now finish the definition of product (\ref{e31}).
If $I_R(A)<\infty$, we define this product by
partitioning $R$ into a finite number of adjacent disjoint
incomplete intervals; this is possible because of Lemma \ref{l3}.
If, on the contrary, $I_R(A)=\infty$, then we define
$M$ as the projection $P_J$ on $N_J$ along $\R_J$. Because of
transversality of $\Sigma$, the projection
$P_J(x)$ is well defined as a unique element
$y\in N_J$ satisfying $x-y\in R_J$.

Let us study basic properties of the map $M(S):2^R\to\Sigma$.
First, for finite sets $S$, we have $M(S)=A_{r_1}\dots A_{r_m}$,
where
$$
S=\{r_1,\dots,r_m\}\quad{\rm and}\quad r_1 < r_2 <\dots < r_m
$$
in the sense of the order on $R$.

\begin{lem}
Let $S_1,S_2\subseteq R$ and $S_1< S_2$, that is, $r_1< r_2$ for
each pair $r_1\in S_1$, $r_2\in S_2$. Then $M(S_1\cup S_2)=M(S_1)M(S_2)$.
\end{lem}
\begin{proof}
Let us consider four possible cases:

\noindent (i) $I_{S_1}(A)<\infty$ and $I_{S_2}(A)<\infty$,

\noindent (ii) $I_{S_1}(A) = \infty$ and $I_{S_2}(A)<\infty$,

\noindent (iii) $I_{S_1}(A) < \infty$ and $I_{S_2}(A) = \infty$,

\noindent (iv) $I_{S_1}(A) = \infty$ and $I_{S_2}(A) = \infty$.

For instance, if
$I_{S_1}(A)<\infty$ and $I_{S_2}(A)=\infty$, we have
$I_S(A)=\infty$ and $M(S)=P_J$. Then the required statement
follows from the easy fact $P_J A_j=P_J$ for any $A_j\in\Sigma$.
The remaining cases are also obvious.
\end{proof}

More generally, suppose $F$ is a monotone map from $R$ to another
linearly ordered set $Q$, that is, $r_2\ge r_1$ implies
$F(r_2)\ge F(r_1)$. Define $M'(q)=M(F^{-1}(Q))$ for each
$q\in Q$.

\begin{thm}
For each subset $S'\subseteq Q$, we have
$$
M'(S')=M(F^{-1}(S')).
$$
\end{thm}
\begin{proof}
Let us use induction on $k=\# J$. It suffices to
consider the cases $I_{F^{-1}(S')}(A)=\infty$
and $I_{F^{-1}(S')}(A)<\infty$. In each case, the required
assertion follows directly from the definitions.
\end{proof}

The following assertion, again, follows from elementary
induction considerations.

\begin{thm}\label{l12}
Any $M(S)$ can be represented as a limit of matrices $M_i\in {\mathcal
L}(\Sigma)$ as $i\to\infty$.
\end{thm}

\begin{thm}\label{t6}
The matrix-valued function $M(S^i(r))$ on $R$ has bounded variation
for $i=1,\dots,4$,
where $S_1(r)=\{p\in R:p\le r\}$,
$S_2(r)=\{p\in R:p < r\}$,
$S_3(r)=\{p\in R:p \ge r\}$,
and $S_4(r)=\{p\in R:p> r\}$.
\end{thm}
\begin{proof}
As is known \cite{VlaSt9912}, there exists an upper bound $V$
on the variation
$$
V(\{A_i\})=\sum_{i=1}^{m-1}\|M_i-M_{i+1}\|,\quad{\rm where\ }%
M_i=A_1\dots A_i,
$$
of any finite product $A_1\dots A_m$, $A_i\in\Sigma$, $i=1,\dots,m$. Let us
consider a finite partition of $R$ into intervals $S_i$, $i=1,\dots,m$, and
prove that the same bound is valid for the product $M(S_1)\dots M(S_m)$.
Indeed, by Theorem \ref{l12}, any $M(S)$ can be approximated by finite
products from ${\mathcal L}(\Sigma)$ with arbitrary precision, and, hence
$$
\sum_{i=1}^{m-1}\|M(S_{i+1})-M(S_i)\| \le V.
$$
\end{proof}

\begin{thm}\label{t4}
The matrix $M(R)$ can be found as an inductive limit of all finite products
$M(G)$, $G\subseteq R$, that is, For each $\e>0$, there exists a finite set
$F\subseteq R$ such that
$$
\|M(F')-M(R)\|<\e \quad{\rm for\ all\ finite\ } F'\supseteq F.
$$
\end{thm}
\begin{proof}
If $I_R(A)=\infty$, this is obvious. If $I_R(A)<\infty$, we
use the induction on $\# J$ again.
\end{proof}

Hence, the product $M(R)$ can also be defined by means of the following
formal constructions. Let ${\mathcal F}$ be the family of all finite subsets
$F\subseteq R$. Considering ${\mathcal F}$ as a directed set with respect to
the order relation $F\le F'\Leftrightarrow F\subseteq F'$, we can define
$M(R)$ as the limit of the net $\{M(F):F\in{\mathcal F}\}$.

\section{Insertions}

The LCP property means that if, at each discrete time instant,
a matrix from $\Sigma$ is added to the current product at the left,
then the resulting procedure converges. For RCP {\em left} should be
replaced by {\em right}. It is also easy to see that, for a CP set $\Sigma$,
one can add matrices alternately at the left and at the right, in
any sequence, and the resulting product still converges.
Now, we are going to prove that, for CP sets, this procedure
can be generalized so that any subsequent
matrix can be inserted at any place of the current product,
the front and the rear positions included.

\begin{thm}\label{t5}
Let a sequence $M_i\in\Sigma^i$, $i=1,\dots$, possess the following
property. For each $i$, there exists an $m_i$, $0\le m_i\le i$ such that
$$
M_i=M_- M_+\quad {\rm and}\quad M_{i+1}=M_- A_{j_i} M_+,
$$
where $M_-\in \Sigma^{m_i}$, $M_+\in \Sigma^{i-m_i}$, and $A_{j_i}\in\Sigma$.
Then $M_i$ converges to some $n\by n$-matrix as $i\to\infty$.
\end{thm}
\begin{proof}
Let us introduce a linear order relation $\succ$ on the set of indices
$1,2,\dots$ as follows. We will write $i'\succ i''$ if the matrix
$A_{j_{i'}}$ takes position to the right of $A_{j_{i''}}$ in the product
$M_i$, where $i=\max\{i',i''\}$.

Let us choose a countable number of reals $x_i$ such that $x_i>x_j$ if and
only if $i\succ j$. This kind of choice is possible because, if at step $i$
the set $\{x_j:j\succ i\}$ satisfies this requirement, the next point
$x_{i+1}$ can also be chosen in a way that the requirement still holds for
the set $\{x_j:j\succ i+1\}$. The assertion of the theorem follows now from
Theorem \ref{t4}, where $R=\{x_i:i=1,2,\dots\}$ equipped with the order
induced by the natural order on $\R$, and $A(x_i)=A_{j_i}$.
\end{proof}

Moreover, the following stronger assertion holds.

\begin{thm}\label{tt4}
There exists a uniform upper bound $V$ on the variation of
any sequence $M_i$ from the hypothesis of Theorem \ref{t5}.
\end{thm}
\begin{proof}
Let us again use induction on $\# J$ and suppose that the assertion\
is proved for all proper subsets of $\Sigma$ (denote the corresponding
upper bound by $B_0$). Let $R$ be the set defined
in the proof of Theorem \ref{t5} and $A(x_i)=A_{j_i}$.
Denote by $R_j$ the subset
$\{x_i:i=1,\dots,j\}$ of $R$. Denote also $k(j)=I_{R_j}(A)$, $j=1,2,\dots$.
The sequence $k(j)$ is nondecreasing and $k(1)=0$ if $\# J>1$.
Now, denote by $l_m$ the maximal index $j$ such that $k(j)=m$
(if it exists).

Let us first find an upper bound on the variation of the finite
sequence $M_1,\dots M_{l_1}$. By Lemma \ref{l3}, the set $R_{l_1}$ can be
partitioned into $3$ incomplete intervals. At each step $i<l_1$,
the matrix $A_{j_i}$ is inserted into one of these subintervals.
By the induction assumption, variation of the product matrix for
each subinterval does not exceed $B_0$ and, hence, because of
(\ref{st1}), the variation of $\{M_1,\dots, M_{l_1}\}$
does not exceed $3B_0$.

Now, let us find an upper bound for the variation of
$$
{\mathcal M}_m=\{M_{l_{m-1}+1},\dots,M_{l_m}\}
$$
for an arbitrary $m$. The set $R_{l_m}$ can be divided into $m$ disjoint
adjacent intervals $S_i$ such that $I_{S_i}(A)=1$, $i=1,\dots,m$. Whenever a
matrix is insereted into one of these intervals, the variation of the whole
product at this step does not exceed $q^{m-1} V'$, where $V'$ is the
variation of the current product, at which the insertion occurs ($q<1$ is the
constant from Theorem \ref{bels}). Thus, the variation of ${\mathcal M}_m$ is
bounded from the above by $s_m=3mq^{m-1}$. But the sum $\sum_{m=1,\dots}s_m$
is bounded since $q<1$, hence, the theorem is proved.
\end{proof}

Let us formulate a generalization of Theorem \ref{tt4}.
The proof is completely analogous to that of Theorem \ref{tt4},
so we leave it to the reader.

\begin{thm}
Let $R$ be a linearly ordered set and let $A$ be a map from $R$
to a CP set $\Sigma$. Suppose $Q$ is another linearly ordered
set and $B$ is a monotone map from $Q$ to $2^R$, that is,
$B(q_2)\supseteq B(q_1)$ whenever $q_2\ge q_1$. Then the map
$M'(q)=M(B(q))$ possesses the bounded variation property
in the sense of Theorem \ref{t6}.
\end{thm}

\section{Linear control systems}

In this section we assume that $R$ has the minimal element $r_-$
and the maximal element $r_+$.
Suppose $f(r)$ is a bounded function from $R$ to $\R^n$.
Let us define the integral
\begin{equation}\label{ee1}
\int_R f(r)dM(r), \quad {\rm where\ } M(r)=M(\{s\in R:s\le r\}),
\end{equation}
as follows. For each finite subset $F\subseteq R$ of the form
$$
F=\{r_1,\dots,r_m\},\quad r_1=r_- <r_2 <\dots < r_{m-1}< r_m=r_+,
$$
define
$$
H(F)=\sum_{i=1}^{m-1} (M(r_{i+1})-M(r_i))f(r_i).
$$
Then define
\begin{equation}\label{ee2}
\int_R f(r)dM(r) = \lim_{F\to R} H(F),
\end{equation}
where the limit for the net of all finite subsets
of $R$ containing $r_-$ and $r_+$ is considered in the right-hand
side. The expression (\ref{ee2}) is well defined because of
Theorem \ref{t6}.

The integral (\ref{ee2}) is an anlogue of the standard
Lebesgue--Stiltjes integral.
Finally, let us define formally

\begin{equation}\label{e32}
\int_R M(r)df(r)=M(r_+)f(r_+)-M(r_-)f(r_-)-\int_R f(r)dM(r).
\end{equation}
The variable integral $x(r)=\int_{R_r} M(s)df(s)$, $r\in R$, can be
interpreted as the output of a linear control system
with the input $f(r)$, $r\in R$.

\section{Recognizing CP property}

Let us say a few words on the general structure of CP families and on the
hardness of their recognition. First of all, the {\em finiteness property}
holds for CP families that have $\rho(\Sigma)=1$. Moreover, if
$\rho(\Sigma)=1$, then there exists a matrix $A\in\Sigma$ such that
$\rho(A)=1$.

Indeed, if this is not the case, there exists a sequence $A_{i}\in\Sigma$
that contains infinite number of entries of at least two different matrices
$A$ and $B$ and such that $\lim_{k\to\infty}A_{k}\dots A_{1}\ne 0$. Then any
non-zero column of the limit matrix is invariant for $A$ and $B$ which is
impossible by assumption. For each $A\in \Sigma$ the limit $\lim_{m\to\infty}
A^{m}$ exists.

Thus we have two options for $\Sigma$. If $\rho(\Sigma)<1$, it is a CP
family. The hardness of recognition of this case is still an open problem as
far as we know. It is conjectured that the problem is algorithmically
unsolvable for matrices with rational entries, the same way as it happens for
a similar problem $\rho(\Sigma)\le 1$, see \cite{BloTh0013}.

The second case is $\rho(\Sigma)=1$ and then the hardness of recognition is
the same as in the first case. Indeed, let $\Sigma=\{A,B,C\}$, where
\[
A=\left(
    \begin{array}{ccc}
      E & 0 & 0 \\
      M_{1} & 0 & 0 \\
      M_{1} & 0 & 0 \\
    \end{array}
  \right),\quad
  B=\left(
    \begin{array}{ccc}
      0 & M_{2} & 0 \\
      0 & E & 0 \\
      0 & 0 & 0 \\
    \end{array}
  \right),\quad
  C=\left(
    \begin{array}{ccc}
      0 & 0 & M_{3} \\
      0 & 0 & 0 \\
      0 & 0 & E \\
    \end{array}
  \right)
\]
and $M_{i}$ are square $n/3\by n/3$-matrices.

Then $\Sigma$ is a CP family if and only if the set $\Psi$ of two matrices
$M_{2}M_{1}$ and $M_{3}M_{1}$ is asymptotically stable, that is, if
$\rho(\Sigma)<1$ (otherwise the variation of paths is not upper bounded).
These are arbitrary matrices of size $n/3\by n/3$, hence the hardness of
recognition is the same as in the first case, but in a space of lower
dimension.

\section{Infinite bounded families}

If $\Sigma$ is an infinite bounded subset of $M_{n}(\R)$, we conjecture that,
again, CP is equivalent to LCP and RCP together. Note that even if we
restrict matrices to orthogonal projections in $\R^{2}$ (they are
self-conjugated, hence, LCP=RCP), infinite families can be either CP or not.

The following assertion can be extended to higher dimensions.
\begin{thm}
A family $\Sigma$ of orthogonal projections on straight lines $\R h$,
$h\in{\mathcal H}\subseteq S_{1}$ in $\R^{2}$ is CP if and only if, for any
sector ${\mathcal K}$ with non-empty interior in $\R^{2}$ there exists a
sector ${\mathcal K}'\subseteq{\mathcal K}$ with non-empty interior such that
${\mathcal K}'\cap {\mathcal H}=\emptyset$.
\end{thm}
The crucial point here is that there are no non-trivial limit cycles for
paths of $\Sigma$.

\end{document}